\documentclass[a4paper,10pt]{article}

\usepackage{amssymb}
\usepackage{amsmath}
\usepackage{amsthm}
\usepackage{multicol}
\usepackage{multirow}

\numberwithin{equation}{section}

\theoremstyle{plain}
\newtheorem{theorem}{Theorem}[section]
\newtheorem{lemma}{Lemma}[section]
\newtheorem{corollary}{Corollary}[section]

\theoremstyle{definition}

\newtheorem{algorithm}{\normalfont\textsc{Algorithm}}

\newtheorem{assumption}{{}}

\newtheorem{assumptionm}{{}}

\newcommand{\norm}[1]{\|#1\|}
\newcommand{\Eproof}{}

\title{Proximal methods for minimizing the sum of a convex function and a composite function}
\author{Tran Dinh Quoc\footnote{Department of Electrical Engineering (ESAT-SCD) and Optimization in Engineering Center (OPTEC), K.U. Leuven, Kasteelpark Arenberg 10, B-3001
Leuven, Belgium.Email: quoc.trandinh,moritz.diehl@ esat.kuleuven.be} {~\LARGE{$\cdot$}}  Moritz Diehl\footnotemark[1]}
\date{}

\begin{document}
\maketitle

\begin{abstract}
This paper extends the algorithm schemes proposed in \cite{Nesterov2007a} and \cite{Nesterov2007b} to the minimization of the sum of a composite objective function and a convex function. Two proximal point-type schemes are provided and their global convergence is investigated. The worst case complexity bound is also estimated under certain Lipschitz conditions and nondegeneratedness. The algorithm is then accelerated to get a faster convergence rate for the strongly convex case. 

\vskip 0.1cm
\noindent\textbf{Keywords: }{Gradient scheme $\cdot$ proximal-type point method $\cdot$ composite function $\cdot$ regularization technique $\cdot$ nonlinear optimization.}
\end{abstract}

\section{Introduction}\label{sec:intro} 
In this paper we consider the following minimization problem:
\begin{equation}\label{eq:NLP}
\min\left\{f(x):= g(x) + \phi(F(x)) ~|~ x\in\Omega \right\}, \tag{P} 
\end{equation}
where $\Omega$ is a nonempty, closed convex subset in $\mathbf{R}^n$, $g :\mathbf{R}^n\to\mathbf{R}$, $F:\mathbf{R}^n\to\mathbf{R}^m$ is continuous differentiable on an open set $\mathcal{F}$ of $\mathbf{R}^n$ and $\phi:\mathbf{R}^m\to\mathbf{R}$ is a proper, lower semicontinuous and convex.

Problem \eqref{eq:NLP} covers many practical problems in optimization, signal processing and statistics. 
For instance, when $F$ reduces to the identity mapping, problem \eqref{eq:NLP} collapes to an optimization problem on a convex set.
If we take $g\equiv 0$ and $\phi(\cdot) = \norm{\cdot}_2^2$ then \eqref{eq:NLP} reduces to a classical least squares problem. 
Moreover, when $g(\cdot) = \norm{\cdot}_{*}$ with a given norm (e.g., Hubber norm, $l_1$-norm) this problem becomes a sparse least squares problem. Sparse least squares problems often appear in signal processing and statistics (see, e.g. \cite{Figueiredo2008,Tropp2006,Wright2007}). 
Another example of \eqref{eq:NLP} is the $l_1$-penalized problem of a nonlinear program, which is represented as 
\begin{equation}\label{eq:L_1_penal}
\min\left\{ p(x) + c\left(\norm{q(x)}_1 + [r(x)]_{+}\right), ~x\in\Omega \right\}, 
\end{equation}
where $p$ is the objective function, $q(x)=0$ is the equality constraints, $r(x)\leq 0$ is then inequality constraints of the original problem, $c>0$ is a penalty parameter, $\Omega$ characterizes for the simple constraints which is assumed to be convex and $[z]_{+} := \sum_{j=1}^l\max\{0,z_j\}$. 
If we define $g := p$, $\phi(u,v) := c\norm{u}_1 + [v]_{+}$ and $F(x) := (q^T, r^T)^T$ then the $l_1$-penalized problem \eqref{eq:L_1_penal} can be reformulated equivalently to \eqref{eq:NLP}.
In some cases, we want to compute the minimum norm solution of a nonlinear systems $F(x) = 0$, then this problem can be reformulated as
\begin{equation}
\min_{x} \rho\norm{x}_2^2 + \norm{F(x)}, 
\end{equation}
where $\rho>0$ is a given parameter. This problem is indeed a particular form of \eqref{eq:NLP}.

The last example in our interest is the problem resulting from nonlinear programming using sharp augmented Lagrangian function \cite{Gasimov2002,Rockafellar1997}. This problem has the following form: 
\begin{equation*}\label{eq:sharp_Lagrange_func}
\min\left\{ H(x,u) := p(x) + u^Tq(x) + c\norm{q(x)}_1 ~|~ x\in\Omega \right\},  
\end{equation*}
where $p$ is the objective function, $q$ is the function of the equality constraints of the original problem, $c>0$ is a penalty parameter.
If we define $g(x) := p(x) + u^Tq(x)$, $\phi(\cdot) := c\norm{\cdot}$ and $F(x) := q(x)$ then this problem becomes a particular case of \eqref{eq:NLP}.  

Let us look at the literature on theory and methods related to problem \eqref{eq:NLP}. They can be roughly classified in two frameworks. The first class is the problem of minimizing the sum of two objective functions and the second one is minimizing a composite objective function. These two classifications, of course, can be theoretically combined in a unified framework as we will see later. However, it is more convenient to exploit the special structure of the problem if we consider it in the form of \eqref{eq:NLP}.
The minimization problems of the sum of two objective functions as well as the minimization of a composite objective function have been investigated early in many research papers. For instance, Mine \textit{et al} \cite{Mine1981} considered the problem of minimizing the sum of two objective functions, where the first function is assumed to be smooth and the other one is assumed to be simple.
Fukushima and Mine \cite{Fukushima1981} then considered the problem of minimizing composite objective function, where the outer function is assumed to be nondifferentiable. A popular case of minimizing a composite objective function is least squares problems, where the outer function is taken of the form $\norm{\cdot}_2^2$.   
This problem class is then extended to the generalized outer function that is assumed to be convex (see, e.g., \cite{Burke1995,Lewis2008,Nesterov2007a,Yamakawa1989}).
Alternatively, there are myriad of research papers consider the minimization problem of the sum of two objective function (see, e.g., \cite{Nesterov2007b,Polyakova1986,Tuy1994,Womersley1985}). The methods for solving this problem have been quite extensively studied.  
For instance, DC (difference of two convex functions) decomposition, splitting backward-forward methods are the methods for solving some sub-classes of this problem. 

In our framework, the algorithm for solving problem \eqref{eq:NLP} on the one hand can be considered as an extension of the gradient schemes that were considered in \cite{Nesterov2007a} for solving nonlinear systems under nonsmooth least squares problems and then \cite{Nesterov2007b} for minimizing composite of two objective functions. 
On the other hand, it can be regarded as a restrict variant of the proximal point algorithm framework in \cite{Lewis2008}. Here, if we define $c(x) := (f(x), F^T(x), x^T)^T$ and $h(u,v, x) := u + \phi(v) + \delta_{\Omega}(x)$ that is convex, where $\delta_{\Omega}(x)$ is the indicator function of the convex set $\Omega$, i.e.
\begin{equation*}\label{eq:indicator_Omega}
\delta_{\Omega}(x) := \begin{cases}0 ~~&\text{if $x\in\Omega$},\\ +\infty ~~&\text{otherwise}, \end{cases} 
\end{equation*}
then problem \eqref{eq:NLP} can be reformulated as 
\begin{equation}\label{eq:Lewis_framework}
\min_{x\in\mathbf{R}^n}h(c(x)). 
\end{equation}
In \cite{Lewis2008}, the authors provided a generic framework so called proximal point method for solving \eqref{eq:Lewis_framework}. The proposed algorithm can be considered as a generalized of the classical proximal point methods introduced by Martinet \cite{Martinet1970}. The theory in this paper is quite general and cover many classes of problems in optimization. This formulation was earlier considered by Burke and Ferris in \cite{Burke1995}.

In this paper, motivated from \cite{Quoc2009b}, a generic algorithm framework so called \textit{sequential convex programming} (SCP) method for solving nonlinear programming, we continue extending the idea of Nesterov in \cite{Nesterov2007a} and \cite{Nesterov2007b} to problem \eqref{eq:NLP}. The main idea of SCP method is to keep the convex substructure of the original problem as much as possible and to convexify the nonconvex part by exploiting the state-of-the-art of the theory and methods in convex optimization.
Let us consider a nonlinear programming problem of the form:
\begin{eqnarray}\label{eq:scp_nlp}
&&\min_{x}p(x) \nonumber\\
[-1.5ex]\\[-1.5ex]
&&\text{s.t.}~ q(x) = 0, ~x\in\Omega,\nonumber 
\end{eqnarray}
where the objective function $p:\mathbf{R}^n\to\mathbf{R}$ and $\Omega\subseteq\mathbf{R}^n$ are assumed to be convex, $g:\mathbf{R}^n\to\mathbf{R}^m$ is assumed to be twice continuously differentiable.
The SCP methods generates an iterative sequence $\{x^k\}_{k\geq 0}$ starting from $x^0\in\Omega$, and computing $x^{k+1}$ by solving the convex subproblem:
\begin{eqnarray}\label{eq:scp_subprob}
&&\min_{x}p(x^k+d) \nonumber\\
[-1.5ex]\\[-1.5ex]
&&\text{s.t.}~ q(x^k) + q'(x^k)d = 0, ~x^k+d\in\Omega,\nonumber 
\end{eqnarray}
where $q'(x^k)$ is the Jacobian matrix of $q$ at $x^k$, to get a solution $d^k$ and setting $x^{k+1} := x^k + \alpha_k d^k$, where $\alpha_k\in (0,1]$ is a given step size.   
 
Note that the convex subproblem \eqref{eq:scp_subprob} may in general have no solution because of the linearized inconsistency. The SCP algorithm may be failed in practice. A popular strategy to handle the linearized inconsistency is to relax the subproblem \eqref{eq:scp_subprob} by introducing slack variables. For instance, this problem can be relaxed as follows:
\begin{eqnarray}\label{eq:scp_subprob_relax}
&&\min_{x}p(x^k+d) + c\sum_{i=0}^m(t_i+s_i)\nonumber\\
&&\text{s.t.}~ q(x^k) + q'(x^k)d = t -s, ~x^k+d\in\Omega,\\
&& ~~~~~ t, s\geq 0,\nonumber 
\end{eqnarray}
where $c>0$ is a penalty parameter. The relaxed problem \eqref{eq:scp_subprob_relax} can be reformulated equivalently to
\begin{eqnarray*}\label{eq:scp_subprob_penalty}
&&\min_{x}p(x^k+d) + c\norm{q(x^k) + q'(x^k)d}_1\nonumber\\
[-1.5ex]\\[-1.5ex]
&&\text{s.t.} ~x^k+d\in\Omega,\nonumber.
\end{eqnarray*}
However, if the objective function $p$ is not strongly convex then the search direction $d^k$ may not be a descent direction. Therefore, a regularization term $\frac{\rho}{2}\norm{d}^2$ should be added to ensure that $d^k$ is a descent direction. The subproblem \eqref{eq:scp_subprob_penalty} now becomes:
\begin{eqnarray*}\label{eq:scp_subprob_penalty_regu}
&&\min_{x}p(x^k+d) + c\norm{q(x^k) + q'(x^k)d}_1 + \frac{\rho_k}{2}\norm{d}^2\nonumber\\
[-1.5ex]\\[-1.5ex]
&&\text{s.t.} ~x^k+d\in\Omega,\nonumber.
\end{eqnarray*}
This problem collapses to the form \ref{eq:subprob2} below.

It has been proved in \cite{Quoc2009b} that under mild conditions, the SCP method converges to a stationary point of the original problem \eqref{eq:scp_nlp} locally in linear rate. The global convergence behaviour of the SCP method has been left unconsidered yet, however. 
        
The aim of this paper is to consider the theoretical aspects of global behaviour of the proximal point methods for solving problem \eqref{eq:NLP} as a bridge connected to the SCP method. The results in this paper is preliminary and should be further considered for the practice purpose.
We first propose a generic algorithm scheme which is based on two different subproblems. We prove some technical results and the convergence of the algorithmic scheme. For the unconstrained case, we are able to provide a worst case global complexity bound under certain conditions. 
When $g$ is strongly convex, the algorithm is accelerated to get a faster global convergence rate that is usually used in gradient methods for convex optimization \cite{Nesterov2004,Nesterov2008a}.    

Throughout the paper, we require the following assumption. 
\begin{assumption}\label{as:A0}
The proper, lower semicontinuous and convex function $\phi$ is Lipschitz continuous on the range space of $F'(x)$, $R_{\phi} := \text{range}F'(x)$, for all $x\in\mathbf{R}^n$ with a global Lipschitz constant $L_{\phi}>0$, i.e.
\begin{equation}\label{eq:phi_lipschitz}
\left|\phi(u) - \phi(v) \right| \leq L_{\phi}\norm{u-v}, ~\forall u, v \in R_{\phi}.
\end{equation}   
\end{assumption}
An example of the function $\phi$ is $\phi(u) = \norm{u}$, where $\norm{\cdot}$ can be taken any norm. For instance, it can be:
\begin{itemize}
\item[i.]   The $l_1$-norm that often appears in the penalty methods, $\phi(\cdot) = \norm{\cdot}_1$;
\item[ii.]  The Euclidean norm frequently used in the Gauss-Newton and the regularization methods, $\phi(\cdot) = \norm{\cdot}_2$;
\item[iii.] The Hubber-norm that is defined by $\phi(u) = \sum_{i=1}^m\sigma(u_i)$, where $\sigma(t)$ is defined as 
\begin{equation*}\label{eq:hubber_norm}
\sigma(t) = \begin{cases} \frac{1}{2}t^2 ~~&\text{if $|t|\leq T$}\\ Tt-\frac{T^2}{2} ~~&\text{otherwise}, \end{cases}
\end{equation*}
where $T > 0$ is given (see, e.g., \cite{Boyd2004}). 
\end{itemize}
Generally, the $\norm{\cdot}$ in the definition of $\phi$ can be any norm, which gives us a freedom choice. Thus, in practice, we can choose the norm $\norm{\cdot}$ such that the Lipschitz constant $L_{\phi}$ as small as possible.

For simplicity of discussion, the Euclidean norm is assumed to be used throughout this paper. We denote by $\nabla g$ is the gradient vector of a scalar function from $\mathbf{R}^n$ to $\mathbf{R}$, $F'$ is the Jacobian matrix of a vector function $F$ from $\mathbf{R}^n$ to $\mathbf{R}^m$.
For a convex function $f:C\to\mathbf{R}\cup\{+\infty\}$, where $C$ is convex set in $\mathbf{R}^n$, $\partial f(x)$ denotes the subdifferential of $f$ at $x$. Each element $\xi\in\partial f(x)$ is called a subgradient of $f$ at $x$.  
The function $f$ is said to be strongly convex with a parameter $\tau>0$ on $C$ if $f(\cdot) - \frac{\tau}{2}\norm{\cdot}$ is convex on $C$ (see, e.g., \cite{Boyd2004}). For a given set $X\subseteq\mathbf{R}$, $\text{int}(X)$ denotes the set of interior point of $X$.

Since problem \eqref{eq:NLP} is nonconvex, a local minimizer (if exist) may not be a global one. A point $x^{*}\in\Omega$ is said to be a stationary (critical) point to problem \eqref{eq:NLP} if
\begin{equation}\label{eq:FNC_NLP}
0 \in \partial g(x^{*}) + F'(x^{*})^T\partial\phi(F(x^{*})) + N_{\Omega}(x^{*}),  
\end{equation}
where $\partial g(x^{*})$ is the subdifferential of $g$ if $g$ is proper, lower semicontinuous and convex, and $\nabla g(x^{*})$, the gradient vector of $g$ if $g$ is differentiable; $F'(x^{*})$ is the Jacobian matrix of $F$ at $x^{*}$ and $F'(x^{*})^T$ is its conjugate operator; $\partial\phi(\bar{F})$ is the subdifferential of $\phi$ at $\bar{F} := F(x^{*})$; and $N_{\Omega}(x^{*})$ is the normal cone of the convex set $\Omega$ at $x^{*}$, i.e.:
\begin{equation*}\label{eq:N_Omega_0}
N_{\Omega}(x^{*}) := \begin{cases} \left\{ w\in\mathbf{R}^n~|~ w^T(y-x^{*}) \geq 0, ~y\in\Omega \right\}, &~\text{if}~ x^{*}\in\Omega,\\
\emptyset, &~\text{otherwise}.\end{cases}
\end{equation*}
Here, we implicitly use the chain rule, which is assumed to be satisfied in our problem setting. The condition \eqref{eq:FNC_NLP} is referred as a necessary optimality condition for \eqref{eq:NLP}. This condition can also expressed as follows:    
\begin{equation*}\label{eq:well_def_cond}
\partial \phi(\bar{F}) \cap \left\{ v~|~ -F'(x^{*})^Tv \in \partial g(x^{*}) + N_{\Omega}(x^{*})\right\} \neq\emptyset, 
\end{equation*}
where $\bar{F}=F(x^{*})$.
Let us denote by $S^{*}$ is the set of critical points of \eqref{eq:NLP} and $S^{*}$ is assumed to be nonempty. 

For a given $x\in\Omega$, we consider the following subproblem:
\begin{equation}\label{eq:subprob1}
\min\left\{ g(x) + \nabla g(x)^Td + \phi(F(x)+F'(x)d) + \frac{\rho}{2}\norm{d}^2 ~|~ x + d\in\Omega \right\}, \tag{$\textrm{P}_1(x)$} 
\end{equation}
where $\rho>0$ is a regularization parameter.
Since $\Omega$ is nonempty, closed and convex, and $\phi$ is proper, lower semicontinuous and convex, this problem has a unique solution.
 
Alternatively, when $g$ is proper, lower semicontinuous and convex, subproblem \ref{eq:subprob1} is slightly changed to:
\begin{equation}\label{eq:subprob2}
\min\left\{ g(x+d) + \phi(F(x)+F'(x)d) + \frac{\rho}{2}\norm{d}^2 ~|~ x + d\in\Omega \right\}. \tag{$\textrm{P}_2(x)$}
\end{equation}
This problem is also strongly convex, which has unique solution. 

A generic algorithm framework for solving problem \eqref{eq:NLP} is briefly described as:
\begin{enumerate}
\item \textit{Initialization}: Choose an initial point $x^0\in\Omega$ and a parameter $\rho_0>0$.
\item \textit{Main iteration}: For each $k=0,1,\dots$. Solve the strongly convex subproblem $\textrm{P}_1(x^k)$ (or $\textrm{P}_2(x^k)$)
to get a unique solution $d^{k}$. If the stopping criterion is not satisfied then set $x^{k+1} := x^k + \alpha_kd^k$ for a given step size $\alpha_k\in (0,1]$, update $\rho_k$ (if necessary) and repeat. 
\end{enumerate}

Particularly, if $\phi$ is identical to zero, subproblem \ref{eq:subprob1} can be considered as a subproblem of classical gradient methods \cite{Nesterov2004} (with regularization term). Alternatively, the subproblem \ref{eq:subprob2} is a subproblem in framework of classical proximal point methods \cite{Martinet1970,Rockafellar1976}.
Note that the subproblem \ref{eq:subprob1} is also closely related to the Levenberg-Marquardt algorithm for solving least squares problems or trust-region methods when the $l_2$-norm is chosen.

The rest of the paper is organized as follows. The next section presents a gradient mapping concept and proves the technical results, which will be used in the sequel. Section \ref{sec:gradient_alg} describes a proximal-point scheme for solving \eqref{eq:NLP} and proves its convergence.
Section \ref{sec:uncon_case} considers the unconstrained cases and provides a global complexity estimate for the previous algorithm.
The last section presents a special case of problem \eqref{eq:NLP}, where the function $g$ is strongly convex with a ``sufficiently large'' parameter. In this case, an accelerating proximal-type scheme is applied to this problem. The global convergence is investigated and the complexity bound is estimated.

\section{Gradient mapping and its properties}\label{sec:grad_method}
Let us first recall some definitions related to the theory in this paper \cite{Boyd2004,Nesterov2004,Nesterov2007b}. As before, for a given convex set $\Omega$, the normal cone of $\Omega$ at $x$ is defined by
\begin{equation}\label{eq:N_Omega}
N_{\Omega}(x) := \begin{cases} \left\{ w\in\mathbf{R}^n~|~ w^T(y-x) \geq 0,~ y\in\Omega \right\}, ~&\text{if}~ x\in\Omega\\
\emptyset, ~&\text{otherwise}.
\end{cases}
\end{equation}
The set of feasible directions to $\Omega$ at $x$ is given by
\begin{equation}\label{eq:F_Omega}
F_{\Omega}(x) := \left\{ d\in\mathbf{R}^n~|~ d = t(y-x), ~y\in\Omega, ~t\geq 0 \right\}, 
\end{equation}
Let us define
\begin{equation}\label{eq:dir_deriv}
Df(x^{*})[d] := \nabla g(x^{*})^Td + \xi(F(x^{*}))^TF(x^{*})d, 
\end{equation}
with $\xi(x^{*}) \in \partial\phi(F(x^{*}))$ the subgradient matrix of $\phi$ at $F(x^{*})$. 
Recall that the necessary optimality condition of problem \eqref{eq:NLP} is 
\begin{equation*}\label{eq:FONC}
0 \in \nabla g(x^{*}) + F'(x^{*})^T\partial \phi(F(x^{*})) + N_{\Omega}(x^{*}).  
\end{equation*}
Then this condition can be expressed equivalently to
\begin{equation}\label{eq:FONC_new}
Df(x^{*})[d] \geq 0, ~~\forall d \in F_{\Omega}(x^{*}).  
\end{equation}

Now, let us define the following mapping
\begin{eqnarray}
&&\psi(y; x ) := f(x) + \nabla f(x)^T(y-x) + \phi(F(x)+F'(x)(y-x)). \label{eq:psi_1}
\end{eqnarray}
Then the convex subproblem \ref{eq:subprob1} can be rewritten as
\begin{eqnarray}
&&f_{\rho}(x) := \min\left\{ \psi(y;x) + \frac{\rho}{2}\norm{y-x}^2 ~|~ y\in\Omega \right\}.\label{eq:subprob1_a}
\end{eqnarray}
Since this problem is uniquely solvable, $f_{\rho}(x)$ is well-defined (finite). Let us denote by $V_{\rho}(x)$ the global solution of this problem, i.e.:
\begin{eqnarray}
&&V_{\rho}(x) =  \text{Arg}\!\min\left\{ \psi(y;x) + \frac{\rho}{2}\norm{y-x}^2 ~|~ y\in\Omega \right\}. \label{eq:V1_rho}
\end{eqnarray}
From these definitions, we have $f_{\rho}(x) = \psi(V_{\rho}(x);x) + \frac{\rho}{2}\norm{x-V_{\rho}(x)}^2$.
The necessary and sufficient optimality condition for subproblem \eqref{eq:subprob1_a} becomes
\begin{eqnarray}
&&\left[\nabla f(x) + \rho(V_{\rho}(x)-x) + F'(x)^T\xi(x)\right]^T(y - V_{\rho}(x)) \geq 0, ~~\forall y\in\Omega, \label{eq:NFO_subprob1_a}
\end{eqnarray}
where $\xi(x) \in \partial\phi(F(x) + F'(x)(V_{\rho}(x)-x))$.
We define a new mapping $G_{\rho}$ as 
\begin{equation}\label{eq:gradient_mapping}
G_{\rho}(x) := \rho(x-V_{\rho}(x)). 
\end{equation}
Then $G_{\rho}$ is referred as a \textit{gradient mapping} of problem \eqref{eq:subprob1_a} (see, e.g., \cite{Nesterov2004,Nesterov2007b}). 

Alternatively, for the subproblem \ref{eq:subprob2}, we define 
\begin{eqnarray}
&&\tilde{\psi}(y;x) := f(y) + \phi(F(x) + F'(x)(y-x))\label{eq:psi_2},\\
&&\tilde{V}_{\rho}(x) := \text{Arg}\!\min\left\{ \tilde{\psi}(y;x) + \frac{\rho}{2}\norm{y-x}^2 ~|~ y\in\Omega \right\},\label{eq:V2_rho}\\
\text{and}~&&\tilde{f}_{\rho}(x) := \tilde{\psi}(\tilde{V}_{\rho}(x);x) + \frac{\rho}{2}\norm{\tilde{V}_{\rho}(x)-x}^2.\label{eq:subprob2_a}    
\end{eqnarray}
The optimality condition for problem \eqref{eq:V2_rho} is
\begin{eqnarray}
&&\left[\nabla f(\tilde{V}_{\rho}(x)) + \rho(\tilde{V}_{\rho}(x)-x) + F'(x)^T\xi(x)\right]^T(y - \tilde{V}_{\rho}(x)) \geq 0, ~~\forall y\in\Omega, \label{eq:NFO_subprob2_a}
\end{eqnarray}
and the gradient mapping $\tilde{G}_{\rho}$ associated with \ref{eq:subprob2} is defined as:
\begin{equation}\label{eq:gradient_mapping}
\tilde{G}_{\rho}(x) := \rho(x-\tilde{V}_{\rho}(x)). 
\end{equation}
Let us denote by $d_{\rho}(x) := V_{\rho}(x) - x$, $\tilde{d}_{\rho}(x) := \tilde{V}_{\rho}(x)-x$, $r_{\rho}(x) := \norm{d_{\rho}(x)}$ and $\tilde{r}_{\rho}(x) := \norm{\tilde{d}_{\rho}(x)}$. The mapping $d_{\rho}$ (resp., $\tilde{d}_{\rho}$) can be considered as a search direction of the proximal algorithm scheme. 
We have the following conclusions.

\begin{lemma}\label{le:stationary_point}
If $V_{\rho}(x) = x$ (resp., $\tilde{V}_{\rho}(x) = x$) then $x$ is a stationary point of \eqref{eq:NLP}. 
\end{lemma}

\begin{proof}
Substituting $V_{\rho}(x) = x$ (resp., $\tilde{V}_{\rho}(x)=x$) into \eqref{eq:NFO_subprob1_a} (resp., \eqref{eq:NFO_subprob2_a}), we again obtain \eqref{eq:FONC}.
\Eproof 
\end{proof}

\begin{lemma}\label{le:V_func}
The norm of gradient mapping $\norm{G_{\rho}(x)}_{*}$ is nondecreasing in $\rho$, and the norm $r_{\rho}(x)$ of the search direction $d_{\rho}(x)$ is nonincreasing in $\rho$. If $g$ is convex then these conclusions also hold for $\norm{\tilde{G}_{\rho}(x)}_{*}$ and  $\tilde{r}_{\rho}(x)$, respectively. Moreover,
\begin{equation}\label{eq:f_f_rho}
f(x) - f_{\rho}(x) \geq \frac{\rho}{2}r_{\rho}^2(x). 
\end{equation}
\end{lemma}

\begin{proof}
It is sufficient to prove the first part. The second part can be proved similarly.
Since the function $k(t,y) := \psi(y;x) + \frac{1}{2t}\norm{y-x}^2$ is convex in two variables $y$ and $t$. We have $\eta(t) := \min_{y\in \Omega}k(t,y)$ is still convex. It is easy to show that $\eta'(t) = -\frac{1}{2t^2}\norm{V_{1/t}(x)-x}^2 = -\frac{1}{2t^2}\norm{d_{1/t}(x)}^2 = \frac{1}{2}\norm{G_{1/t}(x)}^2$.
Since $\eta(t)$ is convex, $\eta'(t)$ is nondecreasing in $t$. This implies that $\norm{G_{1/t}(x)}$ is nonincreasing in $t$. Thus $\norm{G_{\rho}(x)}$ is nondecreasing in $\rho$ and $\norm{d_{\rho}(x)}$ is nonincreasing in $\rho$.

To prove the last inequality \eqref{eq:f_f_rho}, it is implies from the convexity of $\eta$ that
\begin{equation}\label{eq:lm21_eq1}
f(x) = \eta(0) \geq \eta(t) + s(t)(0-t) = \eta(t) + \frac{1}{2t}r^2_{1/t}(x), 
\end{equation}
where $s(t)\in\partial\eta(t)$. On the other hand, $f_{\rho}(x) = \eta(1/\rho)$. Substituting this relation into \eqref{eq:lm21_eq1}, we obtain \eqref{eq:f_f_rho}. The lemma is proved.
\Eproof
\end{proof}

This lemma gives us an observation that if we increase the parameter $\rho$ in the subproblem \ref{eq:subprob1} (resp., \ref{eq:subprob1}) then we will obtain a short search direction $d_{\rho}(x)$ (resp., $\tilde{d}_{\rho}(x)$). Therefore, a suitable choice of the parameter $\rho$ in practice is necessary.

In the sequel, we introduce the following assumptions.
\begin{assumption}\label{as:A1}
The function $F$ is Lipschitz continuously differentiable on $\mathbf{R}^n$ with a Lipschitz constants $L_F$, i.e.
\begin{eqnarray}\label{eq:F_Lipschitz_continuous}
\norm{F'(x) - F'(y)} \leq L_F\norm{y-x}, ~~\forall x, y\in \mathcal{F}. 
\end{eqnarray}
\end{assumption}
\begin{assumption}\label{as:A1b}
The function $g$ is Lipschitz continuously differentiable with a Lipschitz constant $L_g>0$ on $\text{dom}(g)$, i.e. 
\begin{eqnarray}\label{eq:g_Lipschitz_continuous}
\norm{\nabla g(x) - \nabla g(y)} \leq L_g\norm{y-x}, ~~\forall x, y \in\text{dom}g. 
\end{eqnarray}
\end{assumption}
\setcounter{assumptionm}{2}
\begin{assumptionm}\label{as:A1c}
The function $g$ is proper, lower semicontinuous and convex on its domain $\text{dom}(g)$.
\end{assumptionm}

Under the conditions \eqref{eq:F_Lipschitz_continuous} and \eqref{eq:g_Lipschitz_continuous}, applying the mean-valued theorem \cite{Ortega2000} we can easily prove that:
\begin{eqnarray}
&& \norm{F(y) - F(x) - F'(x)(y-x)} \leq \frac{L_F}{2}\norm{y-x}^2 \label{eq:estimate_b},\\
&& \left|g(y) - g(x) - \nabla g(x)^T(y-x)\right| \leq \frac{L_g}{2}\norm{y-x}^2. \label{eq:estimate_a}
\end{eqnarray}
The following lemma show an upper estimation for the objective function $f$ of \eqref{eq:NLP}.

\begin{lemma}\label{le:over_estimate}
Suppose that Assumptions \ref{as:A0} and \ref{as:A1} are satisfied. If, in addition, Assumption \ref{as:A1b} holds then
\begin{eqnarray}
&&\left| f(y) - \psi(y;x) \right| \leq \frac{1}{2}(L_g+L_{\phi}L_F)\norm{y-x}^2. \label{eq:over_estimate1}
\end{eqnarray}  
Alternatively, if, in addition, Assumption \ref{as:A1c} is satisfied then
\begin{eqnarray}
&&\left| f(y) - \tilde{\psi}(y;x) \right| \leq \frac{1}{2}L_{\phi}L_F\norm{y-x}^2. \label{eq:over_estimate2}
\end{eqnarray}
\end{lemma}

\begin{proof} Using the Lipschitz continuity of $\phi$ and estimations \eqref{eq:estimate_a} and \eqref{eq:estimate_b}, we have
\begin{eqnarray}
\left| f(y) - \psi(y;x)\right| &&\leq \left| g(y) - g(x) - \nabla g(x)^T(y-x)\right| \nonumber\\ 
&&+ \left| \phi(F(y)) - \phi(F(x) + F'(x)(y-x))\right| \nonumber\\
&&\leq \frac{(L_g+L_{\phi}L_f)}{2}\norm{y-x}^2,\nonumber
\end{eqnarray}
which proves \eqref{eq:over_estimate1}.
Similarly, we have
\begin{eqnarray}
\left| f(y) - \tilde{\psi}(y;x)\right| &&\leq \left| \phi(F(y)) - \phi(F(x) + F'(x)(y-x))\right| \nonumber\\
&&\leq \frac{L_{\phi}L_f}{2}\norm{y-x}^2,\nonumber
\end{eqnarray}
which proves \eqref{eq:over_estimate2}.
\Eproof
\end{proof}
From these estimations, it follows that
\begin{eqnarray}\label{eq:upper_est}
&&f(y) \leq m_{\rho}(y;x) := \psi(y;x) + \frac{\rho}{2}\norm{y-x}^2,\nonumber\\
[-1.5ex]\\[-1.5ex]
 (\text{resp.,}~ &&f(y)\leq \tilde{m}_{\rho}(y;x) := \tilde{\psi}(y;x)+\frac{\tilde{\rho}}{2}\norm{y-x}^2), ~\forall y\in \Omega,\nonumber 
\end{eqnarray}
provided that $\rho \geq L_g+L_{\phi}L_F$ (resp., $\tilde{\rho} \geq L_{\phi}L_F$). The algorithm scheme is then designed to generate a sequence $\left\{x^k\right\}\subset\Omega$ starting from $x^0\in\Omega$ and decreases the model $m_{\rho}(y;x)$ (resp., $\tilde{m}_{\rho}(y;x)$).

The following lemma provides some useful properties that will be used in the sequel.
\begin{lemma}\label{le:main_estimation}
Under Assumptions \ref{as:A0}-\ref{as:A1}. If Assumption \ref{as:A1b} holds then, for any $x \in\Omega$, we have
\begin{eqnarray}
&&f(x) - f(V_{\rho}(x)) \geq \frac{2\rho - (L_g + L_{\phi}L_F)}{2}r_{\rho}^2(x) = \frac{2\rho -(L_g + L_{\phi}L_F)}{2\rho^2}\norm{G_{\rho}(x)}^2. \label{eq:lm22_est1}\\
&&Df(x)[d_{\rho}(x)] \leq -\rho r_{\rho}(x)^2 = -\frac{1}{\rho}\norm{G_{\rho}(x)}^2. \label{eq:lm22_est2}
\end{eqnarray}
If Assumption \ref{as:A1c} holds then, for any $x \in\Omega$, we have
\begin{eqnarray}
&&f(x) - f(\tilde{V}_{\rho}(x)) \geq \frac{2\rho - L_{\phi}L_F}{2}\tilde{r}_{\rho}(x)^2 = \frac{2\rho - L_{\phi}L_F}{2\rho^2}\norm{\tilde{G}_{\rho}(x)}^2. \label{eq:lm22_est3}\\
&&Df(x)[\tilde{d}_{\rho}(x)] \leq -\rho\tilde{r}(x)^2 = -\frac{1}{\rho}\norm{\tilde{G}_{\rho}(x)}^2. \label{eq:lm22_est4}
\end{eqnarray}
\end{lemma}

\begin{proof}
Let use denote by $V := V_{\rho}(x)$ and $\tilde{V} = \tilde{V}_{\rho}(x)$, from the estimation \eqref{eq:estimate_a}, we have
\begin{eqnarray}
f(V) = g(V) + \phi(F(V)) \leq g(x) + \nabla g(x)^T(V-x) + \frac{L_g}{2}\norm{V-x}^2 + \phi(F(V)).\label{eq:proof_lm22_est1}
\end{eqnarray}
On the other hand, using the optimality condition \eqref{eq:NFO_subprob1_a} for $y=x\in\Omega$, it implies
\begin{equation}\label{eq:proof_lm22_est2}
\nabla g(x)^T(V-x) \leq -\rho\norm{V-x}^2 + \xi(x)^TF(x)(x-V(x)).  
\end{equation}
Since $\phi$ is convex, we
\begin{equation}\label{eq:proof_lm22_est3}
\phi(F(x)) - \phi(F(x)+F'(x)(V-x)) \geq \xi(x)F(x)(x-V), 
\end{equation}
where $\xi(x)\in\partial\phi(F(x)+F'(x)(V-x))$.
Combining \eqref{eq:proof_lm22_est1},\eqref{eq:proof_lm22_est2} and \eqref{eq:proof_lm22_est3}, we obtain
\begin{eqnarray}\label{eq:proof_lm22_est4}
f(V) \leq g(x) + \phi(F(x)) - \frac{2\rho-L_g}{2}\norm{V-x}^2 + \phi(F(V)) - \phi(F(x)+F'(x)(V-x)).  
\end{eqnarray}
Using the $L_{\phi}$-Lipschitz continuity of $\phi$ and estimation \eqref{eq:estimate_b}, we get
\begin{eqnarray}\label{eq:proof_lm22_est5}
\phi(F(V)) - \phi(F(x)+F'(x)(V-x)) &&\leq L_{\phi}\norm{F(V)-F(x)-F'(x)(V-x)}\nonumber\\
[-1.5ex]\\[-1.5ex]
&& \leq \frac{L_{\phi}L_F}{2}\norm{V-x}^2.\nonumber 
\end{eqnarray}
Plugging this inequality into \eqref{eq:proof_lm22_est3}, and noting that $\rho^2\norm{V-x}^2 = \norm{G_{\rho}(x)}^2$, we get \eqref{eq:lm22_est1}.

Now, we prove the second inequality. From the optimality condition \eqref{eq:FONC_new} of \eqref{eq:NLP}, we have 
\begin{eqnarray}\label{eq:proof_lm22_est6}
Df(x)[V-x] = \left[\nabla g(x) + F'(x)^T\bar{\xi}(x)\right]^T(V-x),
\end{eqnarray}
where $\bar{\xi}(x) \in\partial\phi(F(x))$. Using \eqref{eq:NFO_subprob1_a} with $y=x$, it implies
\begin{equation}\label{eq:proof_lm22_est7}
[F'(x)^T\bar{\xi}(x)]^T(V-x) \leq [F'(x)^T(\xi(x) - \bar{\xi}(x))]^T(V-x) - \rho\norm{V-x}^2.  
\end{equation}
By the convexity of $\phi$ we have
\begin{equation}\label{eq:proof_lm22_est8}
[F'(x)^T(\xi(x)-\bar{\xi}(x))]^T(V-x) = -(\bar{\xi}(x)-\xi(x))^TF'(x)(V-x) \leq 0. 
\end{equation}
Combining \eqref{eq:proof_lm22_est6}, \eqref{eq:proof_lm22_est7} and \eqref{eq:proof_lm22_est8}, we get
\begin{equation*}
Df(x)[V-x] \leq -\rho{V-x}^2 = -\frac{1}{\rho}\norm{G_{\rho}(x)}^2, 
\end{equation*}
which proves \eqref{eq:lm22_est2}.

If $g$ is convex then we have
\begin{eqnarray*}
f(\tilde{V}) = g(\tilde{V}) + \phi(F(\tilde{V})) \leq g(x) + \nabla g(\tilde{V})^T(\tilde{V}-x) + \phi(F(\tilde{V})). 
\end{eqnarray*}
Using this inequality and with the same argument as before, we can easily prove the inequalities \eqref{eq:lm22_est3} and \eqref{eq:lm22_est4}.
\Eproof
\end{proof}
The inequality \eqref{eq:lm22_est2} (resp., \eqref{eq:lm22_est4}) shows that $d_{\rho}(x)$ (resp., $\tilde{d}_{\rho}(x)$) is a descent search direction of problem \eqref{eq:NLP}. 

Let us define the level set of the function $f$ restricted to $\Omega$ as follows:
\begin{equation}\label{eq:level_set}
\mathcal{L}_f(\alpha) := \left\{ y\in \Omega ~|~ f(y) \leq \alpha \right\}. 
\end{equation}
We have the following result.

\begin{lemma}\label{le:level_set}
Suppose that $x \in \text{int}(\mathcal{L}_f(f(x))\subseteq\text{int}(\mathcal{F})$. Then if $\rho\geq L_g+L_{\rho}L_F$ (resp., $g$ is convex and $\rho\geq L_{\rho}L_F$) then $V_{\rho}(x)\in \mathcal{L}_f(f(x))$ (resp., $\tilde{V}_{\rho}(x)\in \mathcal{L}_f(f(x))$).  
\end{lemma}
  
\begin{proof}
It is sufficient to prove the first statement. The second one can be done similarly.
It is trivial that $x\in\mathcal{L}_f(f(x))$. Assume, for contradiction, that $V_{\rho}(x)\notin\mathcal{L}_f(f(x))$. Since $x\in\text{int}(\mathcal{L}_f(f(x)))$, the line segment connected $x$ and $V_{\rho}(x)$ insects the boundary of $\mathcal{L}_f(f(x))$ at $x(\bar\alpha) = x + \bar{\alpha}(V_{\rho}(x)-x)$ for some $\bar{\alpha}\in (0,1)$.
From the definition of $x(\bar{\alpha})$ and \eqref{eq:f_f_rho}, we have
\begin{equation}\label{eq:proof_lm23_est1}
f(x(\bar\alpha)) \geq f(x) \geq f_{\rho}(x). 
\end{equation}
Consider $d:=F(x(\bar\alpha)) - F(x) - \bar{\alpha}F'(x)(V-x)$, by virtue of \eqref{eq:estimate_b} with $y=x(\bar\alpha)$, one has
\begin{equation}\label{eq:proof_lm23_est2}
\norm{d} \leq \frac{L_F}{2}\bar{\alpha}^2\norm{V-x}^2. 
\end{equation}
Using Assumptions\ref{as:A1}-\ref{as:A1b} and the convexity of $\phi$ we have.
\begin{eqnarray}\label{eq:proof_lm23_est3}
f(x(\bar\alpha)) &&= g(x+\bar{\alpha}(V-x)) + \phi(F(x)+\bar{\alpha}F'(x)(V-x) + d) \nonumber\\
&&\leq g(x) + \bar{\alpha}\nabla g(x)^T(V-x) + \frac{\bar{\alpha}^2L_g}{2}\norm{V-x}^2 + \phi(F(x)+\bar{\alpha}F'(x)(V-x)) + L_{\phi}\norm{d}  \nonumber\\
&&\leq (1-\bar{\alpha})f(x) + \bar{\alpha}[\psi(V;x) + \frac{\rho}{2}\norm{V-x}^2] - \frac{\bar{\alpha}[2\rho-\bar{\alpha}(L_g+L_{\rho}L_F)]}{2}\norm{V-x}^2.\nonumber
\end{eqnarray}
From \eqref{eq:proof_lm23_est1}, \eqref{eq:proof_lm23_est2} and \eqref{eq:proof_lm23_est3}, we obtain
\begin{equation*}
f(x) \leq f_{\rho}(x) - \frac{[2\rho-\bar{\alpha}(L_g+L_{\rho}L_F)]}{2}\norm{V-x}^2,
\end{equation*}
that is contradict to \eqref{eq:lm22_est1}.
\Eproof 
\end{proof}

\begin{lemma}\label{le:upper_estimate}
Suppose that both $x$ and $V_{\rho}(x)$ in $\mathcal{F}$. Then
\begin{equation}\label{eq:upper_est}
f_{\rho}(x) \leq \min\left\{ f(x) + \frac{\bar{\rho}}{2}\norm{y-x}^2 ~|~ y\in\Omega \right\}, 
\end{equation}
where $\bar{\rho} := \rho+L_g+L_{\phi}L_F$.
Moreover, if $x^{*}$ is a solution to \eqref{eq:NLP} and $\mathcal{L}_f(f(x))\subset\mathcal{F}$ then
\begin{equation}\label{eq:upper_est_star}
f_{\rho}(x) \leq f^{*} + \frac{\bar{\rho}}{2}\norm{x^{*}-x}^2. 
\end{equation}
Alternatively, if both $x$ and $\tilde{V}_{\rho}(x)$ in $\mathcal{F}$ then
\begin{equation}\label{eq:upper_est2}
\tilde{f}_{\rho}(x) \leq \min\left\{ f(y) + \frac{\hat{\rho}}{2}\norm{y-x}^2 ~|~ y\in\Omega \right\}, 
\end{equation}
where $\hat{\rho} := \rho + L_{\phi}L_F$, and if $x^{*}$ is a solution to \eqref{eq:NLP} and $\mathcal{L}_f(f(x))\subset\mathcal{F}$ then
\begin{equation}\label{eq:upper_est_star2}
\tilde{f}_{\rho}(x) \leq f^{*} + \frac{\hat{\rho}}{2}\norm{x^{*}-x}^2. 
\end{equation} 
\end{lemma}

\begin{proof}
For any $y\in\mathcal{F}$, we denote by $d_g(x,y) := g(y) - g(x) - \nabla g(x)(y-x)$ and $d_F(x,y):= F(y) - F(x) - F'(x)(y-x)$. By \eqref{eq:estimate_a} and \eqref{eq:estimate_b}, we have $\norm{d_g(x,y)} \leq \frac{L_g}{2}\norm{y-x}^2$ and $\norm{d_F(x,y} \leq \frac{L_F}{2}\norm{y-x}^2$. 
Since both $x$ and $V_{\rho}(x)$ in $\mathcal{F}$, using the Lipschitz continuity of $\phi$, we have
\begin{eqnarray}
f_{\rho}(x) &&= \min\left\{\psi(y;x) + \frac{\rho}{2}\norm{y-x}^2 ~|~ y\in\Omega \right\} \nonumber\\
&&= \min\left\{ g(y) + d_g(x,y) + \phi(F(y)-d_F(x,y)) + \frac{\rho}{2}\norm{y-x}^2 ~|~ y\in\Omega\right\} \nonumber\\
&&\leq \min\left\{ g(y) + \phi(F(y)) + \frac{\rho+L_g+L_{\phi}L_F}{2}\norm{y-x}^2 ~|~ y\in\Omega \right\}.\nonumber
\end{eqnarray}
By denoting $\bar\rho := \rho+L_g+L_{\phi}L_F$, we obtain \eqref{eq:upper_est}.
The estimation \eqref{eq:upper_est_star} is then proved by substituting $y=x^{*}$ into \eqref{eq:upper_est}.
The remainder is proved similarly.
\Eproof 
\end{proof}

\section{Algorithm framework and its global convergence}\label{sec:gradient_alg} 
Let us denote by $\bar{L} := L_g+L_{\phi}L_F)$ (resp., $\hat{L} := L_{\phi}L_F$). As it is showed in Lemma \ref{le:over_estimate} and \ref{le:upper_estimate} that if the regular parameter $\rho$ is chosen by $\rho \geq \bar{L}$ (resp., $\rho\geq \hat{L}$) then $d_{\rho}(x)$ (resp., $\tilde{d}_{\rho}(x)$ is a descent direction to problem \eqref{eq:NLP}. However, if $\rho$ is to big, the algorithm may generates a short step. Balancing between these issues plays an important role in implementation. 
In the following algorithm, we combine the gradient scheme with a simple line-search strategy to determine $\rho$ adaptively. 

The algorithm is described as follows:

\noindent\rule[1pt]{\textwidth}{1.0pt}{~~}
\begin{algorithm}\vskip -0.3cm\label{alg:A1}{~}\end{algorithm}
\vskip -0.4cm
\noindent\rule[1pt]{\textwidth}{0.5pt}
\noindent{\bf Initialization:} Choose $x^0$ in $\Omega$ and fix $L_0 \in (0, \bar{L}]$ (resp., $L_0\in (0, \hat{L}]$. Set $k:=0$.\\
\noindent{\bf Iteration $k$:} For a given $x^k$, execute the two steps below:
\begin{itemize}
\item[]\textit{Step 1}: Find $\rho_k\in [L_0, 2\bar{L}]$ (resp., $\rho_k\in [L_0, 2\hat{L}]$) such that $f(V_{\rho_k}(x^k)) \leq f_{\rho_k}(x^k)$ (resp., $f(\tilde{V}_{\rho_k}(x^k)) \leq \tilde{f}_{\rho_k}(x^k)$).
\item[]\textit{Step 2}: Update the new iteration $x^{k+1} := V_{\rho_k}(x^k)$ (resp., $x^{k+1} := \tilde{V}_{\rho_k}(x^k)$. Increase $k$ by $1$ and go back to Step 1.
\end{itemize}
\vskip-0.3cm
\noindent\rule[1pt]{\textwidth}{1.0pt}

The computational cost of solving the subproblem \ref{eq:subprob1} (resp., \ref{eq:subprob2}) mostly depends on the structure of the outer convex function $\phi$ and $\Omega$ (resp., additionally, the structure of the function $g$). If $\Omega\equiv\mathbf{R}^n$ (the unconstrained case) and $\phi(u) = \norm{u}$ then the subproblem \ref{eq:subprob1} can be solved by a standard linear algebraic procedure (see \cite{Nesterov2007a}).   

To prove the convergence of Algorithm \ref{alg:A1}, we require the following assumption.
\begin{assumption}\label{as:A2}
The set $\mathcal{F}$ is sufficiently large such that $\mathcal{L}_f(f(x))\subset\mathcal{F}$. 
\end{assumption}
If this assumption is satisfies then $\mathcal{L}_f(f(x^k)) \subset\mathcal{F}$ for all $k\geq 0$ due to the nonincreasing monotonicity of the sequence $\{f(x^k)\}_{k\geq 0}$.

\begin{theorem}\label{thm:main_theorem}
Suppose that Assumptions \ref{as:A0}-\ref{as:A1} hold. Then for any $k\geq 0$:
\begin{itemize}
\item[a)] If Assumption \ref{as:A1b} is satisfied then
\begin{equation}\label{eq:thm31_est1}
f(x^k) - f^{*} \geq \frac{1}{2}L_0\sum_{i=k}^{\infty}r_{\rho_i}(x^i)^2 \geq \frac{L_0}{2}\sum_{i=k}^{\infty}r_{2\bar{L}}(x^i)^2. 
\end{equation}
\item[b)] If Assumption \ref{as:A1c} is satisfied then
\begin{equation}\label{eq:thm31_est2}
f(x^k) - f^{*} \geq \frac{1}{2}L_0\sum_{i=k}^{\infty}\tilde{r}_{\rho_i}(x^i)^2 \geq \frac{L_0}{2}\sum_{i=k}^{\infty}\tilde{r}_{2\hat{L}}(x^i)^2. 
\end{equation} 
\end{itemize}
Consequently, one has
\begin{equation}\label{eq:thm31_est3}
\lim_{k\to\infty}\norm{x^{k+1}-x^k} = 0, 
\end{equation}
and the set of limit points $Q^{*}$ of the sequence $\{x^k\}_{k\geq 0}$is connected. 
If this sequence is bounded (in particular, $L_f(f(x^0))$ is bounded) then every limit point is a critical point of \eqref{eq:NLP}. More further, if the set of limit point $Q^{*}$  is finite then the sequence $\left\{x^k\right\}$ converges to a point $x^{*}$ in $S^{*}$. 
\end{theorem}

\begin{proof}
From Step 1 of Algorithm \ref{alg:A1}, we have $f(x^{i+1}) \leq f_{\rho_i}(x^i)$. Combining this relation and \eqref{eq:f_f_rho}, and note that
$r_{\rho_i}(x^i)$ is nonincreasing in $\rho_i$ by virtue of Lemma \ref{le:upper_estimate}, we have
\begin{equation}\label{eq:thm1_est1}
f(x^{i+1}) \leq f_{\rho_i}(x^i) \leq f(x^i) - \frac{\rho_i}{2}r_{\rho_i}^2(x^i) \leq f(x^i) - \frac{L_0}{2}r_{\rho_i}^2(x^i) \leq f(x^k) - \frac{L_0}{2}r_{2\bar{L}}^2(x^i). 
\end{equation}
Summing up the inequality \eqref{eq:thm1_est1} from $i=k$ to $i=N\geq k$ we get  
\begin{equation}\label{eq:thm1_est2}
f(x^k) - f(x^{N+1}) \geq \frac{1}{2}L_0\sum_{i=k}^Nr_{\rho_i}^2(x^i) \geq \frac{L_0}{2}\sum_{i=k}^{N}r^2_{2\bar{L}}(x^i). 
\end{equation}
Note that the sequence $\{f(x^k)\}_{k\geq 0}$ is bounded from below, passing to the limit as $N\to\infty$ in \eqref{eq:thm1_est2} we obtain \eqref{eq:thm31_est1}. 
The inequalities \eqref{eq:thm31_est2} are proved similarly.

Now, we replace $k=0$ into \eqref{eq:thm31_est1}, it implies that $\sum_{i=0}^{\infty}r_{\rho_i}(x^i)^2 <+\infty$. Since $r_{\rho_i}(x^i) = \norm{x^i-x^{i+1}}$, we get $\lim_{i\to\infty}\norm{x^i-x^{i+1}} = 0$, which proves \eqref{eq:thm31_est3}.

If the sequence $\{x^k\}_{k\geq 0}$ is bounded, by passing to the limit through a subsequence and combining with Lemma \ref{le:stationary_point}, we easily prove that every limit point is a critical point.
If the set of limit points $Q^{*}$ is finite. By applying the result in \cite{Ostrowski1966}[Chapt. 28], we obtain the proof of the remaining conclusion.
\Eproof
\end{proof}

In the framework of least squares problems, it is often that the number of data points is larger than the number of parameters (or variables). In this case, we have $m > n$. 
A critical point $x^{*}\in S^{*}$ of \eqref{eq:NLP} is said to be \textit{nondegenerate} if $\sigma^{*}_F := \sigma_{\min}(F'(x^{*})) > 0$, where $\sigma_{\min}(F'(x^{*}))$ is the smallest singular value of matrix $F'(x^{*})$.
We require the following assumption.

\begin{assumption}\label{as:A3b}
The set of nondegenerate critical points $x^{*}\in S^{*}$ of \eqref{eq:NLP} is nonempty. 
\end{assumption}
We also denote by $\sigma^{*}_g := \sigma_{\min}(\nabla g(x^{*})^T) \geq 0$, the smallest singular value of vector $\nabla g(x)^T$. (This notation is convenient for the case $n=1$).

A set $S^{*}_{\phi}$ is said to be a set of \textit{weak sharp minima} for the function $\phi$ if there exists a constant $\gamma_{\phi} > 0$ such that 
\begin{equation}\label{eq:sharp_minima}
\phi(u) - \phi_{\min} \geq \gamma_{\phi}\text{dist}(u, S^{*}_{\phi}), ~\forall u\in \text{dom}\phi,
\end{equation}
where $\phi_{\min} := \min_{u\in\text{dom}\phi}\phi(u)$ and $\text{dist}(u, S)$ is the Euclidean distance from $u$ to a set $S$.
The constant $\gamma_{\phi}$ and the set $S^{*}_{\phi}$ are called the \textit{modulus} and \textit{domain of sharpness} for $\phi$ over $S^{*}_{\phi}$, respectively (see \cite{Burke1995}).
 
We have following result. 

\begin{theorem}\label{thm:main_theorem2}
Under Assumptions \ref{as:A0}-\ref{as:A2}. Suppose that problem \eqref{eq:NLP} satisfies Assumption \ref{as:A3b} with a local solution $x^{*}\in S^{*}$ and suppose further that the set of weak sharp minima $S^{*}_{\phi}$ of $\phi$ is nonempty. Then, if $x^k\in\mathcal{L}_f(f(x^0))$ and $\norm{x^k-x^{*}} \leq \frac{2\sigma^{*}_f}{3\bar{L}+5\underline{L}}$ then $x^{k+1}\in \mathcal{L}_f(f(x^0))$ and
\begin{equation}\label{eq:thm32_est1}
\norm{x^{k+1}-x^{*}} \leq \frac{3(\bar{L}+\underline{L})\norm{x^k-x^{*}}^2}{2(\sigma_f^{*}-\underline{L}\norm{x^k-x^{*}})} \leq \norm{x^k-x^{*}}. 
\end{equation}
where $\underline{L} := L_g + \gamma_{\phi}L_F$ and $\sigma^{*}_f := \sigma^{*}_g + \gamma_{\phi}\sigma^{*}_F$.
\end{theorem}

\begin{proof}
From \eqref{eq:upper_est_star} and notting that $S^{*}_{\phi}$ is nonempty, using \eqref{eq:sharp_minima}, we have
\begin{eqnarray}\label{eq:thm32_proof1}
\frac{3\bar{L}}{2}\norm{x^k-x^{*}}^2 &&\geq f_{\rho_i}(x^k)-f^{*} \geq \psi(x^{k+1};x^k) -f^{*}\nonumber\\
&& = g(x) - g(x^{*}) + \nabla g(x)^T(x^{k+1}-x^k)\nonumber\\
&& + \phi(F(x^k)+F'(x^k)(x^{k+1}-x^k)) - \phi(F(x^{*})) \\
&&\geq g(x) - g(x^{*}) + \nabla g(x)^T(x^{k+1}-x^k) \nonumber\\
&&+ \gamma_{\phi}\norm{F(x^k)+F'(x^k)(x^{k+1}-x^k) - F(x^{*})}.\nonumber
\end{eqnarray}
Now, using Assumption \ref{as:A2} and \eqref{eq:estimate_a}, we estimate
\begin{eqnarray}\label{eq:thm32_proof2}
g(x^k) \!-\! g(x^{*}) \!+\! \nabla g(x^k)^T(x^{k+1} \!-\! x^k) &&= g(x^k) - g(x^{*}) - \nabla g(x^{*})(x^k-x^{*}) \nonumber\\
&& + [\nabla g(x^k)-\nabla g(x^{*})]^T(x^{k+1}-x^k) + \nabla g(x^{*})(x^{k+1}-x^{*}) \nonumber\\
&&\geq -\frac{L_g}{2}\norm{x^k-x^{*}}^2 - L_g\norm{x^k-x^{*}}\norm{x^{k+1}-x^k}\\
&& + \sigma_g^{*}\norm{x^{k+1}-x^{*}}\nonumber\\
&&\geq [\sigma_g^{*}-L_g\norm{x^k-x^{*}}]\norm{x^{k+1}-x^{*}} - \frac{3L_g}{2}\norm{x^k-x^{*}}^2.\nonumber 
\end{eqnarray}
Similarly, using Assumption \ref{as:A2} and \eqref{eq:estimate_b}, we have
\begin{eqnarray}\label{eq:thm32_proof3}
F(x^k) \!+\! \nabla F(x^k)(x^{k+1} \!\!- \! x^k) \!-\! F(x^{*}) \geq [\sigma_F^{*}\!-\!L_F\norm{x^k\!-\!x^{*}}]\norm{x^{k+1}-x^{*}} \!-\! \frac{3L_F}{2}\norm{x^k-x^{*}}^2. 
\end{eqnarray}
Plugging \eqref{eq:thm32_proof2} and \eqref{eq:thm32_proof3} into \eqref{eq:thm32_proof1} we get
\begin{eqnarray*}
\frac{3\bar{L}+3L_g+3\gamma_{\phi}L_F}{2}\norm{x^k-x^{*}}^2 \geq \left[\sigma_g^{*}+\gamma_{\phi}\sigma_F^{*} - (L_g+\gamma_{\phi}L_F)\norm{x^k-x^{*}}\right]\norm{x^{k+1}-x^{*}}.   
\end{eqnarray*}
Since $\norm{x^k-x^{*}} \leq \frac{\sigma_g^{*}+\gamma_{\phi}\sigma^{*}_F}{L_g+\gamma_{\phi}L_F}$ then the last inequality implies the first part of \eqref{eq:thm32_est1}. If $\norm{x^k-x^{*}} \leq \frac{2(\sigma^{*}_g+\gamma_{\phi}\sigma^{*}_F)}{3\bar{L}+5L_g+5\gamma_{\phi}L_F}$ then we obtain the second part of \eqref{eq:thm32_est1}.
\Eproof
\end{proof}

\section{Global convergence rate of the unconstrained case}\label{sec:uncon_case}
In this section, we consider the rate of global convergence of Algorithm \ref{alg:A1} based on the subproblem \ref{eq:subprob1} for the unconstrained case, i.e. $\Omega \equiv\mathbf{R}^n$.
 
For a given $x\in\mathcal{F}$, let us define the following matrix mapping from $\mathbf{R}^n\to\mathbf{R}^{n\times(m+1)}$: 
\begin{equation}\label{eq:M_matrix}
M(x) := \begin{bmatrix} F'(x)^T & \nabla g(x)\end{bmatrix}_{n\times (m+1)}. 
\end{equation}
The matrix mapping $M(x)$ is said to be \textit{nondegenerate} at $x$ if $\sigma_{\min}(M(x)) \geq \sigma_M > 0$, the smallest singular value of $M(x)$. Matrix $M(x)$ is said to be nodegenerate on a given set $C$ if it is nondegenerate at any $x\in C$.
We make the following assumption.
\begin{assumption}\label{as:A3}
The matrix mapping $M(x)$ is nondegenerate on $\mathcal{L}_f(f(x^0))$.
\end{assumption}
Note that this assumption implies that $m < n$. In term of nonlinear optimization, this is often the case that requires the number of equality constraints is smaller than the number of variables. Assumption \ref{as:A3} is closely related to the linear independent constraint qualification (LICQ) in nonlinear programming.

By using Shur's complement, Assumption \ref{as:A3} is equivalent to $\nabla g(x)\neq 0$ and $\lambda_{\min}(\bar{M}(x)) \geq \sigma_{\bar{M}}^2 > 0$, where $\bar{M}(x) := F'(x)(\norm{\nabla g(x)}^2I_n - \nabla g(x)\nabla g(x)^T)F'(x)^T$ with $I_n$ being the identity matrix. 

\begin{theorem}\label{thm:global_estimate}
Suppose that Assumptions \ref{as:A0}-\ref{as:A1} and \ref{as:A2}-\ref{as:A3} are satisfied and $x^{*}\in S^{*}$ is a critical point of \eqref{eq:NLP}. Then
\begin{itemize}
\item[a)] Let the sequence $\left\{x^k\right\}$ be generated by Algorithm \ref{alg:A1} based on the subproblem \ref{eq:subprob1} and satisfied Assumption \ref{as:A1b}. If $f(x^k) - f(x^{*}) \geq \frac{\sigma_M^2}{2\bar{L}}$, where $\bar{L}:=L_g + L_{\phi}L_F$, then 
\begin{equation}\label{eq:thm42_est1}
f(x^{k+1}) - f(x^{*}) \leq f(x^k) - f(x^{*}) - \frac{\sigma_M^2}{4\bar{L}}. 
\end{equation}
Otherwise,
\begin{equation}\label{eq:thm42_est2}
f(x^{k+1}) - f(x^{*}) \leq \frac{\bar{L}}{\sigma^2_M}[f(x^k)- f(x^{*})]^2. 
\end{equation}
\item[b)] Let the sequence $\left\{x^k\right\}$ be generated by Algorithm \ref{alg:A1} based on the subproblem \ref{eq:subprob1} with $\rho_k = \bar{L}$ and satisfied Assumption \ref{as:A1c}. If $f(x^k) \geq \frac{\sigma_M^2}{\bar{L}}$ then 
\begin{equation}\label{eq:thm42_est1b}
f(x^{k+1}) - f(x^{*}) \leq f(x^k) - f(x^{*}) - \frac{\sigma_M^2}{2\bar{L}}. 
\end{equation}
Otherwise,
\begin{equation}\label{eq:thm42_est2b}
f(x^{k+1}) - f(x^{*}) \leq \frac{\bar{L}}{2\sigma_M^2}[f(x^k)- f(x^{*})]^2 \leq \frac{1}{2}[f(x^k) - f(x^{*})]^2. 
\end{equation} 
\end{itemize}
\end{theorem}

\begin{proof}
It is sufficient to prove the first part 1. The second part is proved similarly.
Suppose that $x^{*}$ is a local minimizer of \eqref{eq:NLP}. Let us consider the linear system
\begin{equation}\label{eq:LNsys}
\begin{cases}F'(x^k)d = 0\\ \nabla g(x^k)^Td + g(x)-g(x^{*}) + \phi(F(x^k)) -\phi(F(x^{*})) = 0. \end{cases} 
\end{equation}
By Assumption \ref{as:A3}, applying Lemma 6 in \cite{Nesterov2007a} with noting that $g(x^k)-g(x^{*})+\phi(F(x^k))-\phi(F(x^{*})) = f(x^k) - f(x^{*}) \geq 0$ , it implies that there exists a solution $d^{*}$ of the linear system \eqref{eq:LNsys} such that
\begin{equation}\label{eq:proof_thm42_est1}
\norm{d^{*}} \leq \frac{g(x^k)-g(x^{*})+\phi(F(x^k))-\phi(F(x^{*})}{\sigma_{\min}(M(x^k))} = \frac{f(x^k)-f(x^{*})}{\sigma_{\min}(M(x^{*}))}. 
\end{equation}
Now, by the rule at Step 1 of Algorithm \ref{alg:A1}, using the convexity of $\phi$ and noting that $\rho_k \leq 2\bar{L}$ (see Step 1 of Algorithm \ref{alg:A1}), where $\bar{L}:=L_g+L_{\phi}L_F$, we have
\begin{eqnarray}\label{eq:proof_thm42_est2}
f(x^{k+1}) &&\leq f_{\rho_k}(x^k) = \min_{d\in\mathbf{R}^n}\left\{\psi(x^k+d;x^k)+\frac{\rho_k}{2}\norm{d}^2\right\}\nonumber\\
&&\leq \min_{t\in [0,1]}\left\{ g(x^k) + t\nabla g(x^k)^Td^{*} + \phi(F(x^k)+tF'(x^k)d^{*}) + \bar{L}t^2\norm{d^{*}}^2\right\} \\
&&\leq \min_{t\in [0,1]}\left\{ (1-t)f(x^k) + t\psi(x^k+d^{*};x^k) + \bar{L}t^2\norm{d^{*}}^2 \right\} \nonumber.
\end{eqnarray}
Since $\phi$ is Lipschitz continuous, and $d^{*}$ is a solution to \eqref{eq:LNsys}, we have
\begin{eqnarray}\label{eq:proof_thm42_est3}
\psi(x^k+d^{*};x^k) &&= g(x^k)+\nabla g(x^k)^Td^{*} + \phi(F(x^k)+F'(x^k)d^{*}) \nonumber\\
&&\leq \nabla g(x^k)^Td^{*} + g(x^k) - g(x^{*}) + \phi(F(x^k)) -\phi(F(x^{*})\nonumber\\
&&+ \phi(F(x^k)+F'(x^k)d^{*}) - \phi(F(x^k)) + g(x^{*}) + \phi(F(x^{*})\\ 
&&\leq L_{\phi}\norm{F'(x^k)d^{*}} + g(x^{*}) + \phi(F(x^{*}) \nonumber\\
&& = f(x^{*}).\nonumber 
\end{eqnarray}
Combining \eqref{eq:proof_thm42_est1}, \eqref{eq:proof_thm42_est2} and \eqref{eq:proof_thm42_est3}, we obtain
\begin{eqnarray}\label{eq:proof_thm42_est4}
f(x^{k+1}) - f(x^{*}) \leq \min_{t\in [0,1]}\left\{ (1-t)[f(x^k)-f(x^{*}] + \frac{\bar{L}}{\sigma_M^2}t^2[f(x^k)-f(x^{*})]^2\right\}.
\end{eqnarray}
Thus if $f(x^k) - f(x^{*}) \geq \frac{\sigma_M^2}{2L}$ then the right hand side of \eqref{eq:proof_thm42_est4} attains the minimum at $t^{*}=\frac{\sigma^2_M}{2L[f(x^k)-f(x^{*}]}$ and therefore, we have 
\begin{eqnarray*}\label{eq:proof_thm42_est5}
f(x^{k+1}) - f(x^{*}) \leq f(x^k) - f(x^{*}) - \frac{\sigma_M^2}{4\bar{L}}. 
\end{eqnarray*}
Otherwise, it attains the minimum at $t^{*}=1$ and we get
\begin{eqnarray*}\label{eq:proof_thm42_est5}
f(x^{k+1}) - f(x^{*}) \leq \frac{\bar{L}}{\sigma_M^2}[f(x^k) - f(x^{*})]^2. 
\end{eqnarray*}
The theorem is proved.
\Eproof
\end{proof}

Let us define 
\begin{equation*}
D(x^0) := \min\left\{ \norm{x^0-x^{*}}, ~~ x^{*}\in S^{*} \right\},
\end{equation*}
the distance from the initial point $x^0$ to the set of stationary points $S^{*}$.
From Lemma \ref{le:upper_estimate}, Algorithm \ref{alg:A1} can guarantee that $f(x^1)-f(x^{*}) \leq \frac{3}{2}\bar{L}D(x^0)^2$, where $\bar{L}:=L_g+L_{\phi}L_F$. 
Now, using Theorem \ref{thm:global_estimate}, it is easy to see that 
\begin{equation*}
N \leq 1 + \frac{2\bar{L}}{\sigma_M^2}(f(x^1)-f(x^{N+1})) \leq 1 + \frac{4\bar{L}}{\sigma_M^2}[f(x^1)-f(x^{*}] \leq 1 + \frac{6\bar{L}^2}{\sigma_M^2}D^2(x^0).  
\end{equation*}
Thus the number of iterations for Algorithm \ref{alg:A1} starting from $x^0$ to enter into the quadratic convergence region is $N_{\min} := 1 + 6\left[\frac{(L_g+L_{\phi}L_F)D(x^0)}{\sigma_M}\right]^2$.  

\section{Accelerated scheme for the strongly convex case.}\label{sec:accelerated_scheme}
When the first term $g(x)$ of the objective function $f(x)$ is strongly convex with a parameter $\tau_{g} \geq L_{\phi}L_F > 0$,
we are able to accelerate Algorithm \ref{alg:A1} by using the same trick as in gradient schemes (see \cite{Nesterov2004,Nesterov2007b}) to solve problem \eqref{eq:NLP}. 
Typically, we require the following assumption.

\begin{assumption}\label{as:A5}
The function $g$ is strongly convex with a parameter $\tau_g$ such that $\tau_g \geq L_{\phi}L_F$. 
\end{assumption}

We consider two sequences $\{a_k\}_{k\geq 0}$ and $\{\varphi_k\}_{k\geq 0}$ generated recursively as follows:
\begin{eqnarray}\label{eq:varphi_k}
&&a_0 := 0, ~~ a_{k+1} := a_{k} + \alpha_{k}, \nonumber\\
&&\varphi_0(x) := \frac{1}{2}\norm{x-x^0}^2, \\
&&\varphi_{k+1}(x) := \varphi_k(x) + \alpha_{k}\left[f(\tilde{V}_{\rho_k}(y^k)) + \frac{1}{\rho_k}\left(\norm{\tilde{G}_{\rho_k}(y^k)}^2 - \tilde{L}\tilde{G}_{\rho_k}(y^k)^T(x-y^k)\right) \right], \nonumber 
\end{eqnarray}
where the sequences $\{\alpha_k\}_{k\geq 0}\subset (0,+\infty)$ and $\{y^k\}_{k\geq 0}$ are given, $\tilde{V}_{\rho_k}(y^k)$ is the solution of \ref{eq:subprob2} with $x=y^k$ and $\rho=\rho_k$, $\tilde{G}_{\rho_k}(y^k) := \rho_k(\tilde{V}_{\rho_k}(y^k)-y^k)$ and $\tilde{L} := \rho_k + L_{\phi}L_F$.

By the construction of $\{a_k\}_{k\geq 0}$ and $\{\varphi_k\}_{k\geq 0}$, it is possible to maintain the following rules for all $k\geq 0$:
\begin{align}
&a_kf(x^k) \leq \varphi_k^{*} := \min_{x\in\Omega}\varphi_k(x), \label{eq:rule_R1}\tag{$R^1_k$}\\
&\varphi_k(x) \leq a_kf(x) + \frac{1}{2}\norm{x-x^0}^2. \label{eq:rule_R2}\tag{$R_k^2$} 
\end{align}
Note that if these rules are maintained then we have 
\begin{equation*}
f(x^k) - f(x^{*}) \leq \frac{\norm{x^0-x^{*}}^2}{2a_k}, ~~ k\geq 1. 
\end{equation*}
Thus by a suitable choice of $\alpha_k$, we can accelerate Algorithm \ref{alg:A1} for this special case.

\begin{lemma}\label{le:fz_fV}
Under Assumption \ref{as:A5}. If $\tilde{V}_{\rho}(x)$ is the unique solution to \ref{eq:subprob2} then
\begin{equation}\label{eq:fz_fV}
f(z) \!-\! f(\tilde{V}_{\rho}(x)) \geq + \frac{1}{\rho}\left[\norm{\tilde{G}_{\rho}(x)}^2 - \tilde{L}\tilde{G}_{\rho}(x)^T(z-x)\right],
\end{equation}
for all $z\in \Omega$, where $\tilde{L} := L_{\phi}L_F + \rho$.
\end{lemma}

\begin{proof}
For simplicity of notation, we denote by $\tilde{V} :=\tilde{V}_{\rho}(x)$. Since $\phi$ is $L_{\phi}$-Lipschitz continuous and convex, using \eqref{eq:estimate_b}, for any $z\in\Omega$, we have
\begin{eqnarray*}\label{eq:lm51_proof1}
\phi(F(z)) - \phi(F(\tilde{V})) && = \phi(F(z)) - \phi(F(x)+F'(x)(z-x)) \nonumber\\
&& + \phi(F(x)+F'(x)(z-x)) - \phi(F(x)+F'(x)(\tilde{V}-x))\nonumber\\
[-1.5ex]\\[-1.5ex]
&&+ \phi(F(x)+F'(x)(\tilde{V}-x)) - \phi(F(\tilde{V}))\nonumber\\
&& \geq -\frac{L_{\phi}L_F}{2}\left[\norm{z-x}^2 +\norm{\tilde{V}-x}^2\right] + (F'(x)^T\tilde{\xi}(x))^T(z-\tilde{V}),\nonumber   
\end{eqnarray*}
where $\tilde{\xi}(x)\in\partial\phi(F(x)+F'(x)(\tilde{V}-x))$. Therefore,
\begin{eqnarray}\label{eq:lm51_proof2}
f(z) - f(\tilde{V}) &&= g(z)+\phi(F(z)) - g(\tilde{V}) - \phi(F(\tilde{V})) \nonumber\\
&&\geq \nabla g(\tilde{V})^T(z-\tilde{V}) + \frac{\tau_g}{2}\norm{z-\tilde{V}}^2\\
&&-\frac{L_{\phi}L_F}{2}\left[\norm{z-x}^2 +\norm{\tilde{V}-x}^2\right] + (F'(x)^T\tilde{\xi}(x))^T(z-\tilde{V}).\nonumber
\end{eqnarray}
Using the optimality condition for \ref{eq:subprob2} we have
\begin{equation*}\label{eq:lm51_proof3}
\nabla g(\tilde{V})^T(z-\tilde{V}) + (F'(x)^T\tilde{\xi}(x))^T(z-\tilde{V}) \geq \rho(\tilde{V}-x)^T(\tilde{V}-z), ~\forall z\in\Omega. 
\end{equation*}
Substituting this inequality into \eqref{eq:lm51_proof2}, we obtain
\begin{eqnarray*}\label{eq:lm51_proof4}
f(z)-f(\tilde{V}) \geq \frac{\tau_g}{2}\norm{z-\tilde{V}}^2 -\frac{L_{\phi}L_F}{2}\left[\norm{z-x}^2+\norm{V-x}^2\right] + \rho(\tilde{V}-x)^T(\tilde{V}-z). 
\end{eqnarray*}
Since $\tau_g\geq L_{\phi}L_F$ by Assumption \ref{as:A5}, the last inequality implies that
\begin{eqnarray}\label{eq:lm51_proof4}
f(z)-f(\tilde{V}) \geq -L_{\phi}L_F(\tilde{V}-x)^T(z-x) - \rho(\tilde{V}-x)^T(z-x) + \rho(\tilde{V}-x)^T(\tilde{V}-x). 
\end{eqnarray}
Substituting $\tilde{G}_{\rho}(x) = \rho(V-x)$ into \eqref{eq:lm51_proof4}, we obtain \eqref{eq:fz_fV}.
\Eproof 
\end{proof}

\begin{corollary}\label{co:check_R2}
Under Assumption \ref{as:A5}. Suppose that the sequence of mappings $\{\varphi_k\}_{k\geq 0}$ defined by \eqref{eq:varphi_k}. Then this sequence maintains the rule \eqref{eq:rule_R2}. 
\end{corollary}

\begin{proof}
We prove by induction. For $k=0$, it is easy to check that the rule \eqref{eq:rule_R2} is true. Assume that this rule holds for some $k\geq 0$. We prove it is true for $k+1$.
Indeed, from the definition \eqref{eq:varphi_k} of $\varphi_k$ and using Lemma \ref{le:fz_fV} with $x=y^k$, we have
\begin{eqnarray}
\varphi_{k+1}(x) && = \varphi_k(x) + \alpha_k\left[f(\tilde{V}_{\rho_k}(y^k)) + \frac{1}{\rho_k}\left(\norm{\tilde{G}_{\rho_k}(y^k)}^2 -\tilde{L}\tilde{G}_{\rho_k}(y^k)^T(z-y^k)\right) \right]\nonumber\\
&&\leq a_kf(x) + \alpha_k\left[f(\tilde{V}_{\rho_k}(y^k)) + \frac{1}{\rho_k}\left(\norm{\tilde{G}_{\rho_k}(y^k)}^2 -\tilde{L}\tilde{G}_{\rho_k}(y^k)^T(z-y^k)\right) \right]  + \frac{1}{2}\norm{x-x^0}^2\nonumber\\
&&\leq a_{k+1}f(x) + \frac{1}{2}\norm{x-x^0}^2\nonumber\\
&& + \alpha_k\left[f(\tilde{V}_{\rho_k}(y^k)) + \frac{1}{\rho_k}\left(\norm{\tilde{G}_{\rho_k}(y^k)}^2 -\tilde{L}\tilde{G}_{\rho_k}(y^k)^T(z-y^k)\right)  - f(x)\right]\nonumber\\
&&\leq a_{k+1}f(x) +  \frac{1}{2}\norm{x-x^0}^2.\nonumber  
\end{eqnarray}
This inequality shows that the rule \eqref{eq:rule_R2} is maintained.
\Eproof 
\end{proof}
Suppose that $v^k$ is the unique solution of the minimization of the function $\varphi_k$ on $\Omega$, i.e.:
\begin{equation}\label{eq:v_k_def}
v_k := \text{arg}\!\min\left\{\varphi_k(x) ~|~ x\in\Omega\right\},  
\end{equation}
and $\varphi^{*}_k := \varphi_k(v^k)$.
We now generate three sequences $\left\{\tau_k\right\}_{k\geq 0}$, $\left\{y^k\right\}_{k\geq 0}$ and $\left\{x^k\right\}_{k\geq 0}$ by the scheme below:
\begin{eqnarray}\label{eq:tau_y_x_k}
&&\tau_k := \frac{\alpha_k}{a_k+\alpha_k} \in (0,1),\nonumber\\
&&y^k := (1-\tau_k)x^k + \tau_kv^k, \\
&& x^{k+1} := \tilde{V}_{\rho_k}(y^k).\nonumber
\end{eqnarray}

The following lemma shows that the rule \eqref{eq:rule_R1} holds for the sequence $\{x^k\}_{k\geq 0}$ defined by \eqref{eq:tau_y_x_k}.

\begin{lemma}\label{le:check_R1}
Under Assumption \ref{as:A5}. Suppose that the sequences $\{\varphi_k\}_{k\geq 0}$ and $\{x^k\}_{k\geq 0}$ defined by \eqref{eq:varphi_k} and \eqref{eq:tau_y_x_k}, respectively. Then
\begin{equation}\label{eq:lm52_rule_R2}
\varphi_{k+1}^{*} \geq a_{k+1}f(x^{k+1}) + \frac{1}{\rho_k^2}\left(a_{k+1}\rho_k - \frac{\alpha_k^2\tilde{L}^2}{2}\right)\norm{\tilde{G}_{\rho_k}(y^k)}^2, 
\end{equation}
where $\tilde{L} := L_{\phi}L_F + \rho_k$ and $\tilde{G}_{\rho_k}(y^k) = \rho_k(x^{k+1}-y^k)$.
Moreover, if $0 < \alpha_k\leq \frac{1}{2}(q_k + \sqrt{q_k^2 + 4q_ka_k})$, where $q_k:=\frac{2\rho_k}{\tilde{L}^2}$, then the rule \eqref{eq:rule_R1} is maintained. 
\end{lemma}

\begin{proof}
We again prove this lemma by induction. For $k=0$ the rule \eqref{eq:rule_R1} is true. Assume that it holds for some $k\geq 0$, we now prove \eqref{eq:rule_R1} holds for $k+1$. For simplicity of notation, we denote by $\tilde{G}^k := \tilde{G}_{\rho_k}(y^k)$.
Note that $\varphi_k$ is strongly convex with parameter $\tau_{\varphi} = 1$, by the assumption of induction, we have
\begin{equation*}\label{eq:lm52_proof1}
\varphi_k(z) \geq \varphi_k^{*} + \frac{1}{2}\norm{z-v^k}^2 \geq a_kf(x^k) + \frac{1}{2}\norm{z-v^{*}}^2, ~~\forall z\in\Omega. 
\end{equation*}
Therefore, using this inequality and Lemma \ref{le:fz_fV} for $z=x^k$, we have
\begin{eqnarray}\label{eq:lm52_proof2}
\varphi_{k+1}^{*} &&= \min_{z\in\Omega}\left\{ \varphi_k(z) + \alpha_k\left[f(x^{k+1}) + \frac{1}{\rho_k}\left(\norm{\tilde{G}^k}^2 -\tilde{L}(\tilde{G}^k)^T(z-y^{k})\right) \right] \right\}\nonumber\\
&& \geq \min_{z\in\Omega}\left\{ a_kf(x^k) + \frac{1}{2}\norm{z-v^k}^2 + \alpha_k\left[f(x^{k+1}) + \frac{1}{\rho_k}\left(\norm{\tilde{G}^k}^2 -\tilde{L}(\tilde{G}^k)^T(z-y^{k})\right) \right] \right\} \nonumber\\
&& \geq a_kf(x^{k+1}) + \frac{a_k}{\rho_k}\left(\norm{\tilde{G}^k}^2 -\tilde{L}(\tilde{G}^k)^T(z-y^{k})\right) + \alpha_kf(x^{k+1}) + \frac{\alpha_k}{\rho_k}\norm{\tilde{G}^{k}}^2\nonumber\\
[-1.5ex]\\[-1.5ex]
&& + \min_{z\in\Omega}\left\{ \frac{1}{2}\norm{z-v^k}^2 - \frac{\alpha_k}{\rho_k}\tilde{L}(\tilde{G}^k)^T(z-y^{k}) \right\} \nonumber\\  
&&= a_{k+1}f(x^{k+1}) + \frac{a_{k+1}}{\rho}\norm{\tilde{G}^k}^2 + \frac{a_k}{\rho_k}\tilde{L}(\tilde{G}^k)^T(y^k-x^{k}) \nonumber\\   
&& + \min_{z\in\Omega}\left\{ \frac{1}{2}\norm{z-v^k}^2 - \frac{\alpha_k}{\rho_k}\tilde{L}(\tilde{G}^k)^T(z-y^{k}) \right\}. \nonumber
\end{eqnarray}
Let us denote the minimization term in the last line of \eqref{eq:lm52_proof2} by $M_k$, then we have 
\begin{eqnarray}\label{eq:lm52_proof3}
M_k &&\geq \min_{z\in\mathbf{R}^n} \left\{ \frac{1}{2}\norm{z-v^k}^2 - \frac{\alpha_k}{\rho_k}\tilde{L}(\tilde{G}^k)^T(z-y^{k}) \right\}\nonumber\\
[-1.5ex]\\[-1.5ex]
&&= -\frac{\alpha_k^2}{2\rho_k^2}\tilde{L}^2\norm{\tilde{G}^k}^2 + \frac{\alpha_k}{\rho_k}\tilde{L}(\tilde{G}^k)^T(y^{k}-v^k).\nonumber
 \end{eqnarray}
Since $a_k(y^k-x^k) + \alpha_k(y^k-v^k) = 0$ by definition \eqref{eq:tau_y_x_k} of $y^k$, plugging this relation and \eqref{eq:lm52_proof3} into \eqref{eq:lm52_proof2} we obtain
\begin{eqnarray*}\label{eq:lm52_proof4}
\varphi_{k+1}^{*} \geq a_{k+1}f(x^{k+1}) + \frac{1}{\rho_k^2}\left(a_{k+1}\rho_k - \frac{\alpha_k^2\tilde{L}^2}{2}\right)\norm{\tilde{G}^k}^2. 
\end{eqnarray*}
The inequality \eqref{eq:lm52_rule_R2} is proved.

Moreover, we note that $0 < \alpha_k\leq \frac{1}{2}(q_k+\sqrt{q_k^2+4q_ka_k})$ then $a_{k+1}\rho_k-\frac{\tilde{L}^2\alpha_k^2}{2} \geq 0$. Hence, $\varphi_{k+1}^{*} \geq a_{k+1}f(x^{k+1})$, i.e. the rule \eqref{eq:rule_R1} holds by induction.
\Eproof 
\end{proof}

According to Lemma \ref{le:check_R1}, the sequence $\{\alpha_k\}_{k\geq}$ has to be chosen such that $0< \alpha_k \leq \frac{q_k+\sqrt{4a_kq_k+q_k^2}}{2}$. 
For simplicity of discussion, in the following algorithm, we choose $\alpha_k := \frac{k+1}{4L_{\phi}L_F}$. The sequence $\{\rho_k\}_{k\geq 0}$ is fixed at $\rho_k=L_{\phi}L_F$ for all $k\geq 0$.
  
The accelerated variant of Algorithm \ref{alg:A1} for solving problem \eqref{eq:NLP} that satisfies Assumption \ref{as:A5} is presented as follows.

\noindent\rule[1pt]{\textwidth}{1.0pt}{~~}
\begin{algorithm}\vskip -0.2cm\label{alg:A2}{~}\end{algorithm}
\vskip -0.3cm
\noindent\rule[1pt]{\textwidth}{0.5pt}
\noindent{\bf Initialization:} Choose $x^0$ in $\Omega$ and fix a parameter $\rho_k := L_{\phi}L_F (:=\hat{L})$ for all $k\geq 0$. 
Set $a_0 :=0 $,  $\varphi_0(z):=\frac{1}{2}\norm{z-x^0}^2$, and $k:=0$.\\
\noindent{\bf Iteration $k$:} For a given $x^k$, execute the four steps below:
\begin{itemize}
\item[]\textit{Step 1}: Compute $v^k$ by solving
\begin{equation}\label{eq:alpha_k}
v^k := \text{arg}\!\min\left\{ \varphi_k(z)~|~ z\in\Omega\right\}. 
\end{equation}
\item[]\textit{Step 2}: Compute $y^k := \frac{k}{k+2}x^k + \frac{2}{k+2}v^k$. 
\item[]\textit{Step 3}: Solve the convex subproblem \ref{eq:subprob2} with $x = y^k$ and $\rho := \rho_k = \hat{L}$ to obtain a unique solution $x^{k+1} := \tilde{V}_{\hat{L}}(y^k)$.
\item[]\textit{Step 4}: Update $\varphi_{k+1}(x)$ by
\begin{equation}\label{eq:varphi_k1}
\varphi_{k+1}(x) :=\varphi_k(x) + \frac{(k+1)}{4\hat{L}}\left[f(x^{k+1}) + \frac{1}{\hat{L}}\norm{\tilde{G}_{\hat{L}}(y^k)}^2 - 2\tilde{G}_{\hat{L}}(y^k)^T(x-y^k) \right]. 
\end{equation}
Increase $k$ by $1$ and go back to Step 1. 
\end{itemize}
\vskip-0.2cm
\noindent\rule[1pt]{\textwidth}{1.0pt}

At the Step 4 of Algorithm \ref{alg:A2}, to update the function $\varphi_k$, the Lipschitz constants $L_{\phi}$ and $L_F$ are required. Otherwise, a line-search strategy should be used to estimate these constants. 
Problem \eqref{eq:alpha_k} at Step 1 is a minimization of a quadratic function on a convex set. The computational cost of solving this problem depends on the complexity of $\Omega$.

The following theorem proves the convergence of Algorithm \ref{alg:A2} and shows that the global complexity bound is $O(\frac{L_{\phi}L_F\norm{x^0-x^{*}}^2}{k^2})$. 

\begin{theorem}\label{thm:convergence_theorem_55}
If the sequence $\{x^k\}_{k\geq 0}$ generated by Algorithm \ref{alg:A2} for solving problem \eqref{eq:NLP} satisfies Assumption \ref{as:A5} then, for $k\geq 1$, we have
\begin{equation}\label{eq:convergence_rate54}
f(x^k) - f(x^{*})  \leq \frac{4L_{\phi}L_F\norm{x^0-x^{*}}^2}{k(k+1)},
\end{equation}
where $x^{*}$ is a stationary point to \eqref{eq:NLP}.
\end{theorem}

\begin{proof}
From the formula of computing $\tau_k$ at Step 1 of Algorithm \ref{alg:A2}, it implies $\alpha_k = \frac{k+1}{4\hat{L}}$ and, as a consequence, $a_k = \sum_{j=0}^{k-1}\alpha_j = \frac{k(k+1)}{8\hat{L}}$. Moreover, we have
\begin{eqnarray}
a_{k+1}\rho_k - \frac{1}{2}\alpha_k^2\tilde{L}^2 = \frac{(k+1)(k+2)}{8} - \frac{(k+1)^2}{32\hat{L}^2}(2\hat{L})^2 = \frac{(k+1)(k+2)}{4}-\frac{(k+1)^2}{4} > 0. \nonumber 
\end{eqnarray}
Therefore, the sequence $\{x^k\}_{k\geq 0}$ generated by Algorithm \ref{alg:A2} satisfies the assumptions of Corollary \ref{co:check_R2} and Lemma \ref{le:check_R1}. Thus the rules \eqref{eq:rule_R1} and \eqref{eq:rule_R2} are maintained. Using these rules, we deduce 
\begin{eqnarray*}
a_kf(x^k) \leq \varphi_k^{*} \leq a_kf(x^{*}) + \frac{1}{2}\norm{x^0-x^{*}}^2. 
\end{eqnarray*}
Consequently, we obtain
\begin{equation*}
 f(x^k) - f(x^{*}) \leq \frac{\norm{x^0-x^{*}}^2}{a_k} = \frac{4\hat{L}\norm{x^0-x^{*}}^2}{k(k+1)}. 
\end{equation*}
The theorem is proved.
\Eproof 
\end{proof}

\vskip 0.2cm
{\small
\noindent{\textbf{Acknowledgments.}}
This research was supported by Research Council KUL: CoE EF/05/006 Optimization in Engineering(OPTEC), GOA AMBioRICS, IOF-SCORES4CHEM, several PhD/postdoc \& fellow grants; the Flemish Government via FWO: PhD/postdoc grants, projects G.0452.04, G.0499.04, G.0211.05, G.0226.06, G.0321.06, G.0302.07, G.0320.08 (convex MPC), G.0558.08 (Robust MHE), G.0557.08, G.0588.09, research communities (ICCoS, ANMMM, MLDM) and via IWT: PhD Grants, McKnow-E, Eureka-Flite+EU: ERNSI; FP7-HD-MPC (Collaborative Project STREP-grantnr. 223854), Contract Research: AMINAL, and Helmholtz Gemeinschaft: viCERP; Austria: ACCM, and the Belgian Federal Science Policy Office: IUAP P6/04 (DYSCO, Dynamical systems, control and optimization, 2007-2011).
}

\bibliographystyle{plain}

\end{document}